\documentclass{amsart}

\usepackage{geometry}
\usepackage{enumerate}
\usepackage{quiver}
\usepackage{xcolor}
\usepackage{tikz-cd}

\usepackage[
    pdfauthor={Casimir Kothari, Joshua Mundinger},
    pdftitle={Dieudonné theory for n-smooth group schemes},
]{hyperref}

\usepackage[
    bibstyle=alphabetic,
    citestyle=alphabetic,
    maxnames=100, 
    doi=false,
    isbn=false,
    giveninits = true, 
]{biblatex}
\usepackage{amsthm}

\theoremstyle{plain}
\newtheorem{theorem}{Theorem}[subsection]
\newtheorem{lemma}[theorem]{Lemma}
\newtheorem{proposition}[theorem]{Proposition}
\newtheorem{corollary}[theorem]{Corollary}

\newtheorem{theoremalpha}{Theorem}

\theoremstyle{definition}
\newtheorem{definition}[theorem]{Definition}
\newtheorem{example}[theorem]{Example}

\newtheorem{remark}[theorem]{Remark}

\usepackage{amsmath,amssymb,mathrsfs}

\newcommand{\F}{\mathbb{F}}
\newcommand{\Fp}{\F_p}
\newcommand{\OO}{\mathcal O}
\newcommand{\Z}{\ensuremath{\mathbb{Z}}}

\DeclareMathOperator{\BT}{BT}  
\DeclareMathOperator{\coker}{coker}

\DeclareMathOperator{\End}{End}
\DeclareMathOperator{\Ext}{Ext}

\DeclareMathOperator{\im}{im}
\DeclareMathOperator{\gr}{gr}
\DeclareMathOperator{\Hom}{Hom}
\DeclareMathOperator{\Homs}{\underline{Hom}}
\DeclareMathOperator{\Lie}{Lie}

\DeclareMathOperator{\Sm}{Sm}
\DeclareMathOperator{\Spec}{Spec}
\DeclareMathOperator{\Spf}{Spf}
\DeclareMathOperator{\op}{op}

\DeclareMathOperator{\modc}{-mod}
\newcommand{\AbR}{\mathrm{Ab}_R}

\newcommand{\G}{\underline{G}}
\newcommand{\M}{\underline{M}}
\newcommand{\Gn}{\G_n}
\newcommand{\Mn}{\M_n}

\newcommand{\ol}[1]{\overline{#1}}
\newcommand{\ul}[1]{\underline{#1}}

\newcommand{\Gm}{\mathbb{G}_m}

\title[
	Dieudonné theory for {\lowercase{$n$}}-smooth group schemes
]{
	Dieudonné theory for 
	{\Large \lowercase{$n$}}-smooth group schemes
}
\author{Casimir Kothari}
\address{Boston University \\ Boston \\ MA}
\email{ckothari@bu.edu}
\author{Joshua Mundinger}
\address{University of California, Berkeley\\ Berkeley\\ CA}
\email{mundinger@berkeley.edu}
\date{\today}

\begin{document}

\begin{abstract}
For all $n \geq 1$, there is a notion of an $n$-smooth group scheme over any $\F_p$-algebra $R$, which may be thought of as a ``Frobenius analogue" of an $n$-truncated Barsotti--Tate group over $R$. We prove that the category of $n$-smooth commutative group schemes over $R$ is equivalent to a certain full subcategory of Dieudonn\'e modules over $R$.  As a consequence, we show that the moduli stack $\Sm_n$ of $n$-smooth commutative group schemes is smooth over $\Fp$ and that the natural truncation morphism $\Sm_{n+1} \to \Sm_n$ is smooth and surjective.  These results affirmatively answer conjectures of Drinfeld. 
\end{abstract}

\maketitle

\section{Introduction}

The goal of this paper is to prove a Dieudonn\'e-theoretic classification for so-called $n$-smooth group schemes over an arbitrary $\F_p$-algebra.  As a consequence, we obtain structural results on the moduli stacks of $n$-smooth group schemes, resolving conjectures of Drinfeld. 

Classification of commutative group schemes in terms of semilinear algebraic data began in 1955, when Dieudonné classified formal Lie groups over a perfect field in terms of modules over what is now known as the Dieudonné ring \cite{Die55}. Since Dieudonné's original papers, many mathematicians have worked to extend this classification to more general groups and base rings. Cartier classified formal Lie groups over a general base ring \cite{Car67Modules}. Grothendieck suggested attacking the problem for $p$-divisible groups by defining a suitable notion of Dieudonné crystal \cite{Gro74}, a task taken up in the following decade \cite{Mes72, MM74, BBM82}. More recently, Gabber and Lau gave a complete classification of commutative finite locally free group schemes of $p$-power order over a perfect $\Fp$-algebra in the spirit of classical Dieudonn\'e theory \cite{Lau13}, and Anschütz, Le Bras, and Mondal classified such group schemes over a quasi-syntomic base using prismatic methods \cite{ALB23, Mon24}.

In this article, we restrict our attention to a particular kind of finite locally free group scheme called an $n$-smooth group scheme, and we prove a Dieudonn\'e-theoretic classification of $n$-smooth commutative group schemes over an arbitrary $\Fp$-algebra. Our formulation and proof are elementary in nature, and make no appeal to crystalline or prismatic Dieudonn\'e theory.  Instead, we use homological properties of $n$-smooth group schemes and Cartier's theory of formal groups. 

Fix a prime number $p$ and let $R$ be an $\Fp$-algebra. 

\begin{definition}\label{definition: n-smooth group}
Let $n \in \mathbb{N}$. An $R$-group scheme $G$ is said to be \emph{$n$-smooth} if Zariski locally on $\Spec R$, there exists an isomorphism of $R$-schemes
\begin{equation*}
G \cong \Spec \frac{R[x_1, \dots, x_r]}{(x_1^{p^n}, \dots, x_r^{p^n})}
\end{equation*}
for some $r \in \mathbb{N}$, under which the identity section of $G$ corresponds to $x_1 = \cdots = x_r = 0$.
\end{definition}  

Such group schemes were studied in \cite[Chapter II.2]{Mes72} and \cite[Chapter VI §2]{Gro74} in connection with truncated Barsotti--Tate groups.  As we will see, one may think of commutative $n$-smooth group schemes as ``Frobenius analogues" of $n$-truncated Barsotti--Tate groups, with formal Lie groups playing the role of $p$-divisible groups.
In particular, if $G/\Spec R$ is $n$-smooth, then $G$ is finite locally free over $R$ and $F^n: G \to G^{(p^n)}$ is zero. 

Let $W(R)$ denote the ring of Witt vectors of $R$, and $\sigma, v: W(R) \to W(R)$ denote the Frobenius and Verschiebung, respectively. The \emph{Cartier--Dieudonné ring} $D_R$ is by definition the associative ring generated by $W(R)$ and two elements $F$ and $V$, subject to the relations $FV = p$ and the relations $Fa = \sigma(a)F, V \sigma(a) = aV$, and  $VaF = v(a)$ for all $a \in W(R)$. 
The last relation implies $VF = v(1)$; since $R$ is an $\Fp$-algebra, $v(1) = p$, so $VF = p$.

\begin{definition}\label{definition: cosmooth module}
Let $n \in \mathbb{N}$.  A left $D_R$-module $M$ is said to be \emph{$n$-cosmooth} if the following conditions are satisfied:
\begin{enumerate}
\item $V^n = 0$ on $M$;
\item $M/VM$ is a finitely generated projective $R$-module;
\item For all $i \in \{1, \dots, n-1\}$, the sequence of abelian groups
\begin{equation*}
\begin{tikzcd}
M \arrow[r, "V^i"] & M \arrow[r, "V^{n-i}"] & M
\end{tikzcd}
\end{equation*}
is exact.
\end{enumerate}
\end{definition}
For a finite locally free commutative group scheme $G/\Spec R$, we let $G^\vee$ denote its relative Cartier dual.

\begin{theoremalpha} \label{thm:equiv}
The functor 
\begin{equation} \label{eqn:equiv}
G \mapsto \Hom(G^\vee, (W_n)_R)
\end{equation}
defines an equivalence of categories between commutative $n$-smooth group schemes over $R$ and $n$-cosmooth $D_R$-modules.
\end{theoremalpha}

\begin{remark}
A $1$-smooth group scheme is a finite locally free group scheme killed by $F$.  Thus Theorem \ref{thm:equiv} in the case $n = 1$ is the well-known classification of height 1 group schemes in terms of restricted Lie algebras (see \cite[Theorem 2.4.2]{Dri24}, which references \cite[Exposé VII.A]{SGA3}).
At the other extreme, formal Lie groups over $R$ are naturally viewed as $\infty$-smooth groups over $R$ (see Lemma \ref{lemma: formal Lie}).  Cartier theory classifies commutative formal Lie groups $G$ over $R$ in terms of the Cartier module $\Hom(G^\vee, W_R)$.  Thus one may view Theorem \ref{thm:equiv} as an interpolation between the classifications for $n = 1$ and $n = \infty$. See §\ref{remark: interpolation with Cartier theory} for more details.
\end{remark}

\begin{remark}
It follows from Theorem \ref{thm:equiv} and Lemma \ref{lemma: adjunction} below that the quasi-inverse to the Dieudonn\'e module functor (\ref{eqn:equiv}) is the functor which associates to an $n$-cosmooth Dieudonn\'e module $M$ the $n$-smooth group scheme whose Cartier dual has functor of points $T \mapsto \Hom_{D_R}(M, W_n(T))$.  Thus our equivalence is just the naive extension of the classical Dieudonn\'e equivalence for unipotent group schemes over a perfect field \cite[Chapter II.4]{Gro74}.  The only difference in our formulation is that since we are over an arbitrary base, we must work with the truncated Witt vectors $W_n$ as opposed to the group ind-scheme $\varinjlim_n W_n$ (see Example \ref{example: Z/p} for further explanation of this point).
\end{remark}

For all $n \geq 1$, let $\Sm_n$ be the moduli stack of commutative $n$-smooth group schemes. It follows from the definition that if $G$ is $(n+1)$-smooth, then the Frobenius kernel $G[F^n]$ is $n$-smooth; this defines a truncation morphism $\tau_n: \Sm_{n+1} \to \Sm_n$.

\begin{theoremalpha} \label{thm:smooth}
For all $n \geq 1$, $\Sm_n$ is a smooth algebraic stack over $\Fp$, and the truncation morphism $\tau_n: \Sm_{n+1} \to \Sm_n$ is smooth and surjective. 
\end{theoremalpha}

The results of Theorems \ref{thm:equiv} and \ref{thm:smooth} affirmatively answer conjectures of Drinfeld \cite{drinfeld-lectures}.

For $n \in \mathbb{N}$, let $\BT_n$ denote the $\Z$-stack of $n$-truncated Barsotti--Tate groups, which is an fpqc algebraic stack locally of finite type over $\Z$ \cite[Proposition 1.8]{Wed01}. 
Theorem \ref{thm:smooth} is the $n$-smooth analogue of Grothendieck's smoothness result, proved in Illusie's article \cite{Ill85}, which states that for all $n \geq 1$, $\BT_n \to \Spec \Z$ is smooth and the truncation morphism $\BT_{n+1} \to \BT_n$ induced by $G \mapsto G[p^n]$ is smooth.

\begin{remark}
It is natural to ask how the equivalence of Theorem \ref{thm:equiv} compares to the classification results of \cite{ALB23} and \cite{Mon24} in the case when the base $R$ is quasi-syntomic. We hope to address this question in future work.
\end{remark}

Let us give some indications about the proofs of Theorems \ref{thm:equiv} and \ref{thm:smooth}.  First, every $n$-smooth group scheme $G$ is an iterated extension of $1$-smooth group schemes, suggesting an inductive approach to the classification result.  We begin by establishing that $\Hom(G^\vee, W_n)$ is an iterated extension of the restricted Lie algebra of $G$ (Corollary \ref{corollary: exact into Wn}). The main step is Lemma \ref{lemma: zero coboundary}, which uses Grothendieck's formula for $\tau_{\leq 1} R\Homs(-,\mathbb{G}_a)$ to deduce that certain $\Ext^1$ classes are zero.  To complete the proof of full faithfulness, we carefully compare the extension structure of $G$ and $\Hom(G^\vee, W_n)$, using the snake lemma to conclude that the unit of a certain adjunction is an isomorphism.

To show essential surjectivity of the Dieudonn\'e module functor, we appeal to Cartier's classification of formal Lie groups by $\infty$-cosmooth modules, which we review in  §\ref{subsection: Cartier}.  First, we show that every $n$-cosmooth module $\ol{M}$ can be Zariski locally lifted to an $\infty$-cosmooth module $M$ (Proposition \ref{proposition: lift n-cosmooth module}).  We then obtain via Cartier theory a formal Lie group $H$ with Cartier module $M$, and $G := H[F^n]$ is then an $n$-smooth group scheme with $\Hom(G^\vee, W_n) \cong \ol{M}$.  

Theorem \ref{thm:smooth} then follows from Theorem \ref{thm:equiv} and Proposition \ref{prop:n-cosmooth local}, which describes a local presentation for $n$-cosmooth modules. 

The structure of this article is as follows.  In Section \ref{section: preliminaries}, we state our conventions and recall some preliminaries on $n$-smooth group schemes, formal Lie groups, and the co-Lie complex.  In Section \ref{section: fully faithful} we set up a general adjunction (Lemma \ref{lemma: adjunction}), establish canonical short exact sequences of Dieudonn\'e modules (Corollary \ref{corollary: exact into Wn}), and use these to prove full faithfulness of the Dieudonn\'e module functor on $n$-smooth group schemes (Theorem \ref{theorem: isomorphism of unit}).  In Section \ref{section: essentially surjective} we prove some foundational results on $n$-cosmooth modules and establish the essential surjectivity of the Dieudonn\'e module functor (Theorem \ref{theorem: essentially surjective}), completing the proof of Theorem \ref{thm:equiv}. In Section \ref{section: base change} we establish compatibility of the Dieudonn\'e functor (\ref{eqn:equiv}) with base change, and make some explicit computations of Dieudonn\'e modules over an arbitrary base. In Section \ref{section: smoothness}, we prove Theorem \ref{thm:smooth}. 

\subsection*{Acknowledgements}
The authors would like to thank
Vladimir Drinfeld, 
Matt Emerton, 
Luc Illusie,
Keerthi Madapusi,
Akhil Mathew,
and
Shubhodip Mondal
for helpful conversations and comments.
Thanks also to the anonymous referee for their comments.
Special thanks to Vladimir Drinfeld for useful comments on an earlier draft of this paper and to Matt Emerton for his guidance and encouragement.

C.K.\ was supported by the National Science Foundation Graduate Research Fellowship under Grant No.\ 2140001. J.M.\ was supported by the National Science Foundation under Award No. 2503534. Any opinions, findings, and conclusions or recommendations expressed in this material are those of the authors and do not necessarily reflect the views of the National Science Foundation.

\section{Preliminaries} \label{section: preliminaries}

\subsection{Notation and conventions}

All group schemes in this article are assumed commutative.

Throughout, $p$ is a fixed prime number and $R$ is a fixed $\F_p$-algebra. If $X/\Spec R$ is a scheme, then the \emph{Frobenius twist} $X^{(p)}$ is defined to be the pullback of $X$ along the absolute Frobenius morphism $\Spec R \to \Spec R$. The relative Frobenius morphism $F_X: X \to X^{(p)}$ is the $R$-linear morphism defined by the diagram
\[\begin{tikzcd}
	X \\
	& {X^{(p)}} & X \\
	& {\Spec R} & {\Spec R}
	\arrow["F_X", from=1-1, to=2-2]
	\arrow[curve={height=-12pt}, from=1-1, to=2-3]
	\arrow[curve={height=12pt}, from=1-1, to=3-2]
	\arrow[from=2-2, to=2-3]
	\arrow[from=2-2, to=3-2]
	\arrow[from=2-3, to=3-3]
	\arrow[from=3-2, to=3-3]
\end{tikzcd}\]
where $X \to X$ is the absolute Frobenius morphism.  We will often write $F: X \to X^{(p)}$ for the relative Frobenius morphism $F_X$. 
As a functor on $R$-algebras, the Frobenius twist is given by $X^{(p)}(S) = X(S_{F})$, where $S_F$ is the $R$-algebra $R \to S \overset{Frob_S}{\to} S$. The relative Frobenius $F: X \to X^{(p)}$ is the natural transformation $(Frob_S)_*: X(S) \to X(S_F)$. The same formula makes sense for any functor from $R$-algebras to sets.

For $n \geq 1$, $F^n: X \to X^{(p^n)}$ denotes the composition of relative Frobenius morphisms $X  \to X^{(p)} \to \cdots \to X^{(p^n)}$.
If $G$ is a group scheme over $\Spec R$, then $F: G \to G^{(p)}$ is a morphism of groups, which is functorial in $G$ (\cite[Expose VII$_A$, 4.2]{SGA3}).

Let $(W_n)_R$ be the ring scheme over $\Spec R$ of length $n$ Witt vectors, and $W_R = \varprojlim_n (W_n)_R$ be the ring scheme of Witt vectors. Let $\sigma: W_R \to W_R$ be the canonical Frobenius lift. Since $W_R$ is obtained via base change from $W_{\F_p}$, we can identify $W_R$ and $W_R^{(p)}$, and under this identification, $\sigma$ is identified with the relative Frobenius $F_{W_R}$. As $R$ is a fixed $\F_p$-algebra, we will omit subscripts in the sequel and write $W_n$ and $W$ in place of $(W_n)_R$ and $W_R$, respectively.

The Frobenius lift $\sigma:W(R) \to W(R)$ extends to a morphism $\sigma: D_R \to D_R$ which sends $V \mapsto V$ and $F \mapsto F$. If $M$ is a left $D_R$-module, then $M|_{\sigma^i}$ denotes the restriction of scalars of $M$ along $\sigma^i: D_R \to D_R$.

\subsection{Cartier duality and Verschiebung}

If $G /\Spec R$ is a finite locally free commutative group scheme, its relative Cartier dual, denoted $G^\vee$, is the group scheme $\ul{\Hom}(G, \mathbb{G}_m)$.  Concretely, the Hopf algebra of functions of $G^\vee$ is the $R$-linear dual of $\OO(G)$. Cartier duality induces an involutive anti-equivalence from the category of finite locally free $R$-group schemes to itself.

More generally, the functor $\Homs(-,\mathbb{G}_m)$ is fully faithful on the category of affine commutative group schemes $G/\Spec R$ whose coordinate ring is flat Mittag-Leffler \cite[Theorem A]{AM25}. The class of flat Mittag-Leffler modules (whose definition we do not require) contains all projective modules, and countably generated flat Mittag-Leffler modules are projective. The essential image of this functor consists of group ind-finite ind-schemes whose coordinate pro-ring is pro-projective, and the inverse is given by $\Homs(-,\mathbb{G}_m)$. 

This setup applies to formal Lie groups. A formal Lie group is of the form $\Spf B$ where $B$ is a topological Hopf $R$-algebra whose underlying $R$-module is locally of the form $\prod_I R$ with the product topology.
The topological linear dual $B^\vee$ is then a Hopf $R$-algebra whose underlying $R$-module is locally free \cite[Corollary 3.1.2]{AM25}, and we have isomorphisms 
\[ \Homs(\Spf B,\Gm) \cong \Spec B^\vee, \qquad \Homs(\Spec B^\vee, \Gm) \cong \Spf B\]
\cite[Theorem 3.1.5]{AM25}.
In particular, Cartier duality is fully faithful on formal Lie groups. 

For all flat group schemes $G/\Spec R$, there exists a Verschiebung homomorphism $V_G: G^{(p)} \to G$ (\cite[Expose VII$_A$, 4.3]{SGA3}); if the group $G$ is understood, we write $V$ for $V_G$.  The homomorphism $V_G$ is functorial in $G$, and we have $V_G \circ F_G = [p]: G \to G$ and $F_G \circ V_G = [p]: G^{(p)} \to G^{(p)}$. Moreover, when $G$ is finite, $V_G$ agrees with the Cartier dual of the Frobenius morphism $F_{G^\vee}: G^\vee \to (G^\vee)^{(p)}$. 
Furthermore, we can concretely identify the Verschiebung operator on Witt vectors with the usual shift operator:

\begin{lemma}[\cite{Dem72} III.3 Proposition, p. 57]
\label{lemma: Witt V is Verschiebung}
Under the identification $(W_n)^{(p)} \cong W_n$, $V_{W_n}$ is identified with the shift homomorphism $(x_1, \dots, x_n) \mapsto (0, x_1, \dots, x_{n-1})$.
\end{lemma}

\subsection{$n$-smooth group schemes} \label{subsection: n-smooth}

Recall the definition of $n$-smooth group schemes from Definition \ref{definition: n-smooth group}: an affine $R$-group scheme $G$ is $n$-smooth if its ring of functions is locally isomorphic to $R[x_1,\ldots, x_r]/(x_1^{p^n},\ldots, x_r^{p^n})$ where the identity section corresponds to $x_1 = \cdots = x_r = 0$. The following proposition makes precise the idea that $n$-smooth group schemes are ``Frobenius analogues" of $n$-truncated Barsotti--Tate groups (cf.\ characterizations of $\BT_n$ groups in \cite[Remark 1.3]{Ill85}).

\begin{lemma}[\cite{Gro74}, Chap. VI, Proposition 2.1]\label{lemma: exact sequences for n-smooth}
    Let $n$ be a positive integer and $G/\Spec R$ be a group scheme. Then the following are equivalent:
    \begin{enumerate}[(a)]
        \item $G$ is $n$-smooth;
        \item $G$ is finite locally free, $F_G^n = 0$, and the sequence 
        \begin{equation}\label{eq: F exact sequences}
        \begin{tikzcd}
	       G & {G^{(p^i)}} & {G^{(p^n)}}
	       \arrow["{F^i}", from=1-1, to=1-2]
	       \arrow["{F^{n-i}}", from=1-2, to=1-3]
        \end{tikzcd}\end{equation}
        is exact\footnote{A sequence of affine $R$-group schemes is said to be exact if the corresponding sequence of fppf sheaves over $\Spec R$ is exact.} for all $i \in \{1,\ldots, n-1\}$;
        \item $G$ is finite locally free, $F^nG = 0$, and the sequence \eqref{eq: F exact sequences} is exact for some $i \in \{1,\ldots, n-1\}$. 
        \end{enumerate}
\end{lemma}

If $G/\Spec R$ is an $n$-smooth group scheme, then it follows from the definition that the co-Lie algebra $\omega_G := e^* \Omega^1_{G/R}$ is a vector bundle over $\Spec R$, where $e: \Spec R \to G$ is the identity section.  By duality, $\Lie(G) = \Hom(G^\vee, \mathbb{G}_a)$ is a vector bundle over $R$ of the same rank.  

\begin{definition} \label{definition: rank of n-smooth}
The \emph{rank} of an $n$-smooth group scheme $G$ over $\Spec R$ is the rank of $\Lie(G)$, a locally constant function on $\Spec R$. 
\end{definition}

\begin{definition}\label{definition: cosmooth group}
A group scheme $G/\Spec R$ is \emph{$n$-cosmooth} if it is finite locally free and its Cartier dual $G^\vee$ is $n$-smooth. 
\end{definition}

\begin{example}[Group schemes of order $p$]
    Over an algebraically closed field of characteristic $p$, the group schemes of order $p$ are $\mu_p$, $\Z/p\Z = \mu_p^\vee$, and $\alpha_p$ \cite[Lemma 1]{TO70}.
    \begin{itemize}
        \item $\mu_p$ is 1-smooth but not 1-cosmooth.
        \item $\Z/p\Z$ is 1-cosmooth but not 1-smooth.
        \item $\alpha_p$ is 1-smooth and 1-cosmooth.
    \end{itemize}
\end{example}

By definition, the \textit{rank} of an $n$-cosmooth group scheme $G$ is the rank of its Cartier dual $G^\vee$ in the sense of Definition \ref{definition: rank of n-smooth}.  Cartier duality induces an anti-equivalence of categories between $n$-smooth and $n$-cosmooth group schemes over $\Spec R$.  Since the equivalence in Theorem \ref{thm:equiv} factors through Cartier duality, we will work primarily with $n$-cosmooth group schemes.

For $r \geq 1$, let $\Sm_n^r/\Fp$ denote the moduli prestack of commutative $n$-smooth group schemes of rank $r$.

\begin{lemma} \label{lemma: finite type algebraic stack}
$\Sm_n^r$ is an fpqc algebraic stack of finite type over $\Fp$.
\end{lemma}

\begin{proof}
Let $\mathcal{X}_n^r$ denote the stack of finite locally free group schemes of order $p^{nr}$.  Then $\Sm_n^r \subset \mathcal{X}_n^r$ is a locally closed substack of $\mathcal{X}_n^r$: by \cite[Lemma 2.2.3]{Dri24} and upper semicontinuity of the rank of the Lie algebra, $\Sm_n^r$ is an open substack of the closed substack of $\mathcal{X}_n^r$ where $F^n = 0$.  Since $\mathcal{X}_n^r$ is an fpqc algebraic stack of finite type over $\Fp$ (as can be seen via the presentation obtained from fixing a basis of the Hopf algebra of a group scheme of order $p^{nr}$), we obtain the lemma.
\end{proof}

We let $\Sm_n$ denote the $\Fp$-stack of $n$-smooth group schemes. It is the disjoint union of $\Sm_n^r$ over all $r \geq 1$.

\subsection{The relation to formal Lie groups}
From the definition, it follows that if $G$ is $N$-smooth and $n \leq N$, then $G[F^n]$ is $n$-smooth. Thus, $G \mapsto G[F^n]$ defines the \emph{truncation morphism} $\Sm_N \to \Sm_{n}$, making the stacks $\{\Sm_n\}_{n \geq 1}$ into a projective system. Moreover, if $G$ is a formal Lie group over $R$, then \cite[Theorem 2.1.7]{Mes72} shows that $G$ is ``$F$-divisible" as an fppf sheaf of groups, and the definitions show that $G[F^n]$ is $n$-smooth for any $n \geq 1$. Letting $\Sm_\infty$ denote the stack of commutative formal Lie groups over $\Fp$, we have the following:

\begin{lemma}[\cite{Dri24}, 2.2.10 and \cite{Mes72}, Chapter II] \label{lemma: formal Lie}
The assignment $G \mapsto \{G[F^n]\}_{n \geq 1}$ defines an equivalence 
\begin{equation*}
\Sm_\infty \simeq \varprojlim_n \Sm_n.
\end{equation*}
\end{lemma}

So we think of $n$-smooth group schemes as truncations of formal Lie groups. In fact, we will show in this article that every $n$-smooth group scheme is, Zariski locally on the base, obtained as the kernel of $F^n$ on a formal Lie group (see Remark \ref{remark: n-smooth locally as kernel of Fn}).

\subsection{The co-Lie complex}

We recall the co-Lie complex of a group scheme, following \cite[§2.1]{Ill85}. 

\begin{definition}
Let $S$ be a scheme and $G$ be a finite locally free $S$-group scheme.  The \emph{co-Lie complex} of $G$ is the complex
\begin{equation*}
\ell_G := Le^* L_{G/S} \in D(S)
\end{equation*}
where $L_{G/S}$ is the cotangent complex of $G$ and $e: S \to G$ is the identity section. 
\end{definition}

The co-Lie complex $\ell_G$ is a perfect complex of amplitude in $[-1, 0]$.  The following result is well-known \cite[Remark 2.1.4]{Mes72}.

\begin{lemma} \label{lemma: n-smooth coLie}
Let $R$ be an $\Fp$-algebra. If $G/\Spec R$ is $n$-smooth, then $H^{-1}(\ell_G)$ and $H^0(\ell_G)$ are finite locally free $R$-modules of rank equal to the rank of $G$. 
\end{lemma}
\begin{proof}
This follows immediately from the definitions, as every $n$-smooth group scheme of rank $r$ is locally isomorphic to $\Spec R[x_1, \dots, x_r]/(x_1^{p^n}, \dots, x_r^{p^n})$ where the identity section corresponds to $x_1 = \dots = x_r = 0$.
\end{proof}

The following result of Grothendieck will be crucial for cohomological calculations with $n$-smooth group schemes. 

\begin{lemma}[Grothendieck, \cite{MM74} §14]\label{lemma: grothendieck hom formula}
    If $S$ is a scheme, $G$ is a finite locally free commutative $S$-group scheme, and $M$ is a quasi-coherent $\OO_S$-module, then 
    \[
        R\Homs_{\OO_S}(\ell_G,M) \simeq \tau_{\leq 1} R\Homs(G^\vee, M)
    \]
    where the latter $R\Homs$ is in the category of fppf sheaves of abelian groups over $S$.
\end{lemma}

\section{Full Faithfulness of the Dieudonn\'e Functor} \label{section: fully faithful}
In this section, we establish that the Dieudonn\'e functor $G \mapsto \Hom(G, W_n)$ is fully faithful on $n$-cosmooth group schemes over $R$.  Our strategy is to construct an adjoint to the Dieudonn\'e functor and to show that the unit of this adjunction is an isomorphism, using key homological properties of $n$-cosmooth group schemes.

\subsection{The Dieudonné adjunction}

Let $\AbR$ be the category of fppf sheaves of abelian groups over $\Spec R$. The Witt vectors $W$ are a left $D_R$-module in $\AbR$.
\begin{lemma}\label{lemma: VD is a two-sided ideal}
    $D_R V \subseteq V D_R.$
\end{lemma}
\begin{proof}
    $D_R$ is generated by $W(R)$, $F$, and $V$.
    The relations imply $FV = p = VF$ and $aV = V\sigma(a)$ for $a \in W(R)$.
\end{proof}
Thus, $V^nW \subseteq W$ is a subsheaf of $D_R$-modules and $W_n = W/ V^nW$ inherits a $D_R$-module structure from $W$.

\begin{definition} \label{definition: level n Dieudonne}
    The \emph{level $n$ Dieudonné functor} $\Mn: \AbR^{op} \to (D_R\modc)$ is defined by 
    \[\Mn(A) = \Hom_{\AbR}(A, W_n).\] 
\end{definition}

Define the functor 
\[ \Gn: (D_R\modc)^{op} \to \AbR\]
by $\Gn(M)(S) = \Hom_{D_R}(M, W_n(S))$ for all $R$-algebras $S$.  We will sometimes denote this as $\Gn(M) = \Hom_{D_R}(M, W_n)$.

\begin{lemma}
	$S \mapsto \Hom_{D_R}(M,W_n(S))$ is an fppf sheaf over $\Spec R$.
\end{lemma}
\begin{proof}
	$S \mapsto W_n(S)$ is a sheaf of $D_R$-modules.
	Since $\Hom_{D_R}(M,-)$ preserves limits, it preserves the sheaf condition.
\end{proof}

\begin{remark}
By writing down an explicit representing algebra, one can see that $\Gn(M)$ is an affine group scheme over $R$, which is of finite type if $M$ is finitely generated. We will not use this observation in what follows.
\end{remark}

\begin{lemma}\label{lemma: adjunction}
	$\Mn$ and $\Gn$ are adjoint functors.
\end{lemma}
\begin{proof}
	We wish to construct an isomorphism
	\[ \Hom_{\AbR}(A, \Gn(M)) \cong \Hom_{D_R}(M, \Mn(A))\]
	natural in $A \in \AbR$ and $M \in D_R\modc$.
	If $B$ is any sheaf of associative rings, $X$ is a sheaf of abelian groups, and $Y$ and $Z$ are sheaves of left $B$-modules, then the adjunction of the tensor product and the internal Hom functor $\Homs$ gives an isomorphism
	\[ \Hom(X, \Homs_B(Y,Z)) \cong \Hom_B(Y \otimes X, Z) \cong \Hom_B(Y, \Homs(X,Z)),\]
	natural in $X,Y,$ and $Z$. In our case, we take $B$ to be the constant sheaf on $D_R$, $Y$ to be the constant sheaf on $M$, and $Z$ to be $W_n$.
\end{proof}

\subsection{Homological properties of $n$-cosmooth group schemes}

Let $G$ be an $n$-cosmooth group scheme over $R$. The main goal of this subsection is Corollary \ref{corollary: exact into Wn}, which establishes a canonical exact sequence exhibiting $\Hom(G, W_n)$ as an extension of $\Hom(G, W_{n-j})$ by $\Hom(G, W_j)$ for all $j \in \{1, \dots, n-1\}$.  Using this, we conclude that $\Hom(G, W_n)$ is an $n$-cosmooth $D_R$-module in the sense of Definition \ref{definition: cosmooth module}.

First, it is a general fact that maps into $W_j$ factor through the cokernel of $V^j$:

\begin{lemma}\label{lemma: maps into smaller Witt}
    Let $H /\Spec R$ be a flat group scheme. If $i \geq j$, the quotient map $H \to \coker(V_H^i)$ induces an isomorphism 
    \[ \Hom(\coker(V^i_H),W_j) \overset{\sim}{\to} \Hom(H,W_j).\]
\end{lemma}
\begin{proof}
    Since $H \to \coker(V^i_H)$ is surjective, the map is injective.
    Now suppose $f: H \to W_j$ is given. By naturality of the Verschiebung, 
    $f V_H^i = V_{W_j}^i f^{(p^i)}$, and $V_{W_j}^i = 0$ for $i \geq j$.
    Thus $f$ factors through $H \to \coker(V_H^i)$, as desired.
\end{proof}

Fix $i \in \{1, \dots, n-1\}$. Let $Q_i = (G^\vee[F^i])^\vee = \coker(V_G^i)$, a quotient of the $n$-cosmooth group scheme $G$, and let $K_{n-i}$ be the kernel of $G \to Q_i$. We have a short exact sequence 
\begin{equation}\label{eq: canonical cosmooth quotients}
	\begin{tikzcd}
	0 & {K_{n-i}} & G & {Q_i} & 0
	\arrow[from=1-1, to=1-2]
	\arrow[from=1-2, to=1-3]
	\arrow[from=1-3, to=1-4]
	\arrow[from=1-4, to=1-5]
\end{tikzcd},\end{equation}
and by Lemma \ref{lemma: exact sequences for n-smooth}, $Q_i$ is $i$-cosmooth and $K_{n-i}$ is $(n-i)$-cosmooth. 
Since $G$ is $n$-cosmooth, we have an identification $K_{n-i} \cong Q_{n-i}^{(p^i)}$. Note also that the ranks of $G, Q_i$, and $K_{n-i}$ agree.

\begin{lemma}\label{lemma: zero coboundary}
	If $G/\Spec R$ is an $n$-cosmooth group scheme, then the coboundary map 
	\[ \Hom(G,W_{n-1}) \to \Ext^1(G,\mathbb{G}_a)\]
	induced by the short exact sequence $0 \to \mathbb{G}_a \to W_{n} \to W_{n-1} \to 0$
	is zero.
\end{lemma}
\begin{proof}
    Consider the commutative square 
\[
\begin{tikzcd}
	{\Hom(G,W_{n-1})} & {\Ext^1(G,\mathbb{G}_a)} \\
	{\Hom(K_1,W_{n-1})} & {\Ext^1(K_1,\mathbb{G}_a)}
	\arrow[from=1-1, to=1-2]
	\arrow[from=1-1, to=2-1]
	\arrow[from=1-2, to=2-2]
	\arrow[from=2-1, to=2-2]
\end{tikzcd}
\]
    induced by the short exact sequence $0 \to \mathbb{G}_a \to W_n \to W_{n-1} \to 0$.
	By Lemma \ref{lemma: maps into smaller Witt}, the map $G \to Q_{n-1}$ induces an isomorphism $\Hom(Q_{n-1},W_{n-1}) \cong \Hom(G, W_{n-1})$,
	so restriction along $K_1 \to G$ induces the zero map $\Hom(G,W_{n-1}) \to \Hom(K_1,W_{n-1})$. We will show $\Ext^1(G,\mathbb{G}_a) \to \Ext^1(K_1,\mathbb{G}_a)$ is an isomorphism, from which the lemma follows.
    
    Consider the fiber sequence of co-Lie complexes $\ell_{K_{1}^\vee} \to \ell_{G^\vee} \to \ell_{Q_{n-1}^\vee} \to^{+1}$. The associated long exact sequence of cohomology groups is of the form
\[\begin{tikzcd}
	0 & {H^{-1}(\ell_{K_{1}^\vee})} & {H^{-1}(\ell_{G^\vee})} & {H^{-1}(\ell_{Q_{n-1}^\vee})} \\
	& {H^0(\ell_{K_{1}^\vee})} & {H^0(\ell_{G^\vee})} & {H^0(\ell_{Q_{n-1}^\vee})} & 0
	\arrow[from=1-1, to=1-2]
	\arrow["\ast", from=1-2, to=1-3]
	\arrow[from=1-3, to=1-4]
	\arrow["\ast"{description}, from=1-4, to=2-2]
	\arrow[from=2-2, to=2-3]
	\arrow["\ast", from=2-3, to=2-4]
	\arrow[from=2-4, to=2-5]
\end{tikzcd},\]
    where all the groups involved are locally free $R$-modules of the same rank by Lemma \ref{lemma: n-smooth coLie}. Climbing the exact sequence shows that the maps marked by $\ast$ are isomorphisms; in particular, $H^{-1}(\ell_{K_1^\vee}) \to H^{-1}(\ell_{G^\vee})$ is an isomorphism.
    Now apply the Grothendieck formula $\tau^{\leq 1} R\Homs(A,\mathbb{G}_a) \cong R\Homs(\ell_{A^\vee},\OO)$ of Lemma \ref{lemma: grothendieck hom formula}. If $A$ is a $j$-cosmooth group scheme, then $\ell_{A^\vee}$ has locally free cohomology groups by Lemma \ref{lemma: n-smooth coLie}, so $\Ext^i(A,\mathbb{G}_a) \cong \Hom_R( H^{-i}(\ell_{A^\vee}), R)$ for $i \in \{0,1\}$. Thus $\Ext^1(G,\mathbb{G}_a) \to \Ext^1(K_1,\mathbb{G}_a)$ is an isomorphism.
\end{proof}

\begin{corollary}\label{corollary: exact into Wn}
	If $G/\Spec R$ is an $n$-cosmooth group scheme, then for all $j \in \{1,\ldots, n-1\}$, the short exact sequence $0 \to W_j \to W_n \to W_{n-j} \to 0$ induces a short exact sequence
	\[ 0 \to \Hom(G, W_j) \to \Hom(G,W_n) \to \Hom(G,W_{n-j})\to 0.\]
\end{corollary}
\begin{proof}
    Since $\Hom$ is left exact, it suffices to show that $\Hom(G,W_n) \to \Hom(G,W_{n-j})$ is surjective.
    By induction, it suffices to show $\Hom(G, W_i) \to \Hom(G,W_{i-1})$ is surjective for $i \leq n$.
    By Lemma \ref{lemma: maps into smaller Witt}, it suffices to check when $i = n$, in which case surjectivity follows from Lemma \ref{lemma: zero coboundary}.
\end{proof}

Recall from Definition \ref{definition: cosmooth module} that a left $D_R$-module $M$ is $n$-cosmooth if $V^nM = 0$, $\ker(V^i) = \im(V^{n-i})$ for all $i \in \{1,\ldots, n-1\}$, and $M/VM$ is a finitely generated projective $R$-module.
\begin{definition} \label{definition: rank of n-cosmooth module}
The \emph{rank} of an $n$-cosmooth module $M$ is the rank of $M/VM$ as an $R$-module.
\end{definition}

The rank of an $n$-cosmooth group scheme $G$ was defined in §\ref{subsection: n-smooth} to be the rank of the Lie algebra of $G^\vee$.  Using Corollary \ref{corollary: exact into Wn}, we can then establish that the Dieudonn\'e modules of $n$-cosmooth group schemes are $n$-cosmooth $D_R$-modules of the appropriate rank.

\begin{corollary} \label{corollary: image is n-cosmooth}
If $G/\Spec R$ is an $n$-cosmooth group scheme, then $\Mn(G)$ is an $n$-cosmooth $D_R$-module. The rank of $G$ agrees with the rank of $\Mn(G)$.
\end{corollary} 

\begin{proof}
This is a direct application of Corollary \ref{corollary: exact into Wn}. Let $M = \M_n(G)$. Clearly $V^n = 0$ on $M$. Moreover, by Corollary \ref{corollary: exact into Wn}, $M/VM = \Hom(G, \mathbb{G}_a) = \Lie(G^\vee)$.
By Lemma \ref{lemma: n-smooth coLie}, $M/VM$ is a finitely generated projective $R$-module and the rank of $G$ is the rank of $M/VM$.  Finally, the exactness of Corollary \ref{corollary: exact into Wn} is the statement that $\im(V^j) = \ker(V^{n-j})$ on $M$ for all $j \in \{1, \dots, n-1\}$, as desired.
\end{proof}

\subsection{The Dieudonné functor is fully faithful}

Corollary \ref{corollary: exact into Wn} implies that if $G$ is $n$-cosmooth, then $\Mn(G)$ is an iterated extension of the restricted Lie algebra of $G^\vee$. We now prove that $\Mn$ is fully faithful on $n$-cosmooth group schemes inductively using this structure of an iterated extension.

\begin{theorem}\label{theorem: isomorphism of unit}
For all $n \geq 1$, the functor $\Mn$ is fully faithful on $n$-cosmooth group schemes over $\Spec R$.
\end{theorem}

\begin{proof}
    By the adjunction of $\Gn$ and $\Mn$ of Lemma \ref{lemma: adjunction}, 
    it suffices to show that if $G/\Spec R$ is $n$-cosmooth, then the unit map $G \to \Gn\Mn(G)$ is an isomorphism.

	The proof is by induction on $n$. The case $n=1$ is the classical correspondence between restricted Lie algebras and 1-smooth group schemes \cite[Exposé VII.A]{SGA3}.
	Now suppose $G$ is $(n+1)$-cosmooth.
	We have a short exact sequence of $D_{R}$-module schemes 
	\[ 0 \to (\mathbb{G}_a)|_{\sigma^{n}} \to W_{n+1} \to W_{n} \to 0,\]
	where $(-)|_{\sigma}$ means restriction of scalars along Frobenius and the map $(\mathbb{G}_a)|_{\sigma^n} \to W_{n+1}$ is the inclusion of the image of $V^n$.
	By Corollary \ref{corollary: exact into Wn},
	we obtain a short exact sequence of $D_R$-modules
	\[
		0 \to \M_1(G)|_{\sigma^n} \to \M_{n+1}(G) \to \M_n(G) \to 0.
	\]
	Now apply $\underline{G}_{n+1}$ to obtain an exact sequence 
	\begin{equation} \label{eq: Gn on ses}
		0 \to \G_{n+1}\M_n(G) \to \G_{n+1}\M_{n+1}(G) \to \G_{n+1}(\M_1(G)|_{\sigma^n}).
	\end{equation}
	Let $G \to Q_n$ be the cokernel of $V^n$ and let $G \to Q_1$ be the cokernel of $V$.  Then $Q_n^{(p)} = G^{(p)}/\ker(V) \cong \im(V: G^{(p)} \to G)$. We will compare \eqref{eq: Gn on ses} to the exact sequence $ 0 \to Q_n^{(p)} \to G \to Q_1 \to 0$.
	
	We first construct a commutative diagram 	
	\begin{equation}\label{eq: square for sub of G}
		\begin{tikzcd}[cramped]
		0 & {Q_n^{(p)}} & G \\
		0 & {\G_{n+1}\M_n(G)} & {\G_{n+1}\M_{n+1}(G)}
		\arrow[from=1-1, to=1-2]
		\arrow[from=1-2, to=1-3]
		\arrow["\cong"{description}, from=1-2, to=2-2]
		\arrow[from=1-3, to=2-3]
		\arrow[from=2-1, to=2-2]
		\arrow[from=2-2, to=2-3]
	\end{tikzcd}\end{equation}
	where $G\to \G_{n+1}\M_{n+1}(G)$ is the unit map.
	Observe $V^n \M_n(G) = 0$, so 
	\begin{align*} 
		\G_{n+1}\M_n(G) &= \Hom_{D_R}(\Hom(G,W_n),W_{n+1}) \\
		&= \Hom_{D_R}(\Hom(G,W_n),W_{n+1}[V^n]) \\
		&\cong \Hom_{D_R}(\Hom(G,W_n),W_n|_{\sigma}).
	\end{align*}
    Since $\sigma: W_n(S)\to W_n(S)$ is induced by the Frobenius of $S$, 
    it follows that $W_n|_\sigma = W_n^{(p)}$ as sheaves of $D_R$-modules.
    Thus
    \[
        \Hom_{D_R}(\Hom(G,W_n), W_n|_\sigma) \cong \Hom_{D_R}(\Hom(G,W_n),W_n^{(p)}) = \G_n\M_n(G)^{(p)}.
    \]
    Now if $f \in \M_{n+1}(G) = \Hom(G, W_{n+1})$, then by naturality of the Verschiebung, $Vf^{(p)} = f V$, so the following square commutes:
    \[
\begin{tikzcd}[ampersand replacement=\&]
	{Q_n^{(p)}} \& G \\
	{W_n^{(p)}} \& {W_{n+1}}
	\arrow[from=1-1, to=1-2]
	\arrow["{f^{(p)} \mod V^n}"', from=1-1, to=2-1]
	\arrow["f"', from=1-2, to=2-2]
	\arrow[from=2-1, to=2-2]
\end{tikzcd}
.
\]    
Hence the following diagram commutes
\[
\begin{tikzcd}[ampersand replacement=\&]
	{Q_n^{(p)}} \&\& G \\
	{\G_n\M_n(G)^{(p)}} \& {\G_{n+1}\M_n(G)} \& {\G_{n+1}\M_{n+1}(G)}
	\arrow[from=1-1, to=1-3]
	\arrow[from=1-1, to=2-1]
	\arrow[from=1-3, to=2-3]
	\arrow["\sim", from=2-1, to=2-2]
	\arrow[from=2-2, to=2-3]
\end{tikzcd}
\]
where the vertical arrows are the unit maps.
    By induction, the composition $Q_n^{(p)} \to \G_{n+1} \M_{n}(G)$ is an isomorphism.

	Now consider the map $\G_{n+1}\M_{n+1}(G) \to \G_{n+1}(\M_1(G)|_{\sigma^n})$.
	Similarly, $V M_1(G) = 0$, so 
	\begin{align*}
		\G_{n+1}(\M_1(G)|_{\sigma^n}) &= \Hom_{D_R}(\Hom(G,(\mathbb{G}_a)|_{\sigma^n}),W_{n+1}) \\
  &= \Hom_{D_R}(\Hom(G,(\mathbb{G}_a)|_{\sigma^n}),W_{n+1}[V]) \\
		&= \Hom_{D_R}(\Hom(G,(\mathbb{G}_a)|_{\sigma^n}),(\mathbb{G}_a)|_{\sigma^n}).
	\end{align*}
	Now Lemma \ref{lemma: maps into smaller Witt} implies $\M_1(Q_1) \to \M_1(G)$ is an isomorphism.
	Further, since every $D_R$-linear map is also linear after restriction of scalars along $\sigma^n$, we have an injection 
	\[ 
		\G_1\M_1(G)= \Hom_{D_R}(\Hom(G,\mathbb{G}_a),\mathbb{G}_a)\to \Hom_{D_R}(\Hom(G,(\mathbb{G}_a)|_{\sigma_n}),(\mathbb{G}_a)|_{\sigma^n}) = \G_{n+1}(\M_1(G)|_{\sigma^n})
	\]
	Thus we have a commutative square 
\[\begin{tikzcd}[cramped]
	G & {Q_1} & 0 \\
	& {\G_1\M_1(G)} \\
	{\G_{n+1}\M_{n+1}(G)} & {\G_{n+1}(\M_1(G)|_{\sigma^n})}
	\arrow[from=1-1, to=1-2]
	\arrow[from=1-1, to=3-1]
	\arrow[from=1-2, to=1-3]
	\arrow[from=1-2, to=2-2]
	\arrow[hook', from=2-2, to=3-2]
	\arrow[from=3-1, to=3-2]
\end{tikzcd}.\]
	We claim that there exists a map 
	\[ T: \G_{n+1}\M_{n+1}(G) \to \G_1\M_1(G)\]
	which makes the following diagram commute:
\begin{equation}\label{eq: pentagon for quotient of G}
	\begin{tikzcd}[cramped]
	G & {Q_1} & 0 \\
	& {\G_1\M_1(G)} \\
	{\G_{n+1}\M_{n+1}(G)} & {\G_{n+1}(\M_1(G)|_{\sigma^n})}
	\arrow[from=1-1, to=1-2]
	\arrow[from=1-1, to=3-1]
	\arrow[from=1-2, to=1-3]
	\arrow[from=1-2, to=2-2]
	\arrow[hook', from=2-2, to=3-2]
	\arrow["T", dashed, from=3-1, to=2-2]
	\arrow[from=3-1, to=3-2]
\end{tikzcd}
\end{equation}
	If $\psi \in \G_{n+1}\M_{n+1}(G)$ and $f \in \M_1(G)$, then define $T\psi(f) \in \mathbb{G}_a$ as follows:
	by Corollary \ref{corollary: exact into Wn}, there exists $\tilde f \in \M_{n+1}(G)$ lifting $f \in \M_1(G)$ along $\M_{n+1}(G) \to \M_1(G)$. Now $T\psi(f)$ is defined to be the image of $\psi(\tilde f)$ under $W_{n+1} \to \mathbb{G}_a$.
	This is well-defined: lifts $\tilde f$ of $f$ are defined up to $V$, and $\psi$ is $V$-linear.

	Now we show \eqref{eq: pentagon for quotient of G} commutes. The quadrilateral commutes by definition of the unit morphisms. To address the triangle, observe that the map $\G_1\M_1(G) \to \G_{n+1}(\M_1(G)|_{\sigma^n})$ takes $\psi \in \G_1\M_1(G)$ and $f \in \Hom(G,\mathbb{G}_a)|_{\sigma^n}$ to $(0,\ldots,0,\psi(f)) \in V^nW_{n+1}$. But $(0,\ldots,0,f) = V^n \tilde f$ for any lift $\tilde f: G \to W_{n+1}$ of $f: G \to \mathbb{G}_a$. Thus, if $\tilde \psi \in \G_{n+1}\M_{n+1}(G)$ and $f \in \M_1(G)|_{\sigma^n}$,  then 
    \[ (T\tilde \psi) ((0,\ldots,0,f)) = 
       \tilde \psi(V^n \tilde f) = V^n \tilde \psi(\tilde f)  
    \]
    as $\tilde \psi$ is $V$-linear.
    Thus the triangle commutes.

	Joining \eqref{eq: square for sub of G} and \eqref{eq: pentagon for quotient of G} gives a diagram 
\[\begin{tikzcd}[cramped]
	0 & {Q_{n}^{(p)}} & G & {Q_1} & 0 \\
	0 & {\G_{n+1}\M_n(G)} & {\G_{n+1}\M_{n+1}(G)} & {\G_1\M_1(G)}
	\arrow[from=1-1, to=1-2]
	\arrow[from=1-2, to=1-3]
	\arrow[from=1-2, to=2-2]
	\arrow[from=1-3, to=1-4]
	\arrow[from=1-3, to=2-3]
	\arrow[from=1-4, to=1-5]
	\arrow[from=1-4, to=2-4]
	\arrow[from=2-1, to=2-2]
	\arrow[from=2-2, to=2-3]
	\arrow["T"', from=2-3, to=2-4]
\end{tikzcd}\]
	The first row is exact. Moreover, 
	since $\G_1\M_1(G) \to \G_{n+1}(\M_1(G)|_{\sigma^n})$ is injective, the kernel of $T$ is the kernel of $\G_{n+1}\M_{n+1}(G) \to \G_{n+1}(\M_1(G)|_{\sigma^n})$. Since \eqref{eq: Gn on ses} is exact, the second row is exact.
	The first and third vertical maps are isomorphisms by induction.
	The snake lemma shows $G\to \G_{n+1}\M_{n+1}(G)$ is an isomorphism, as desired.
\end{proof}	

\section{Essential Image} \label{section: essentially surjective}

As outlined in the introduction, our strategy for calculating the essential image of the Dieudonné functor $\Mn$ is to compare the Dieudonné modules of $n$-smooth group schemes with the Cartier modules of formal Lie groups.  The key technical result is Proposition \ref{prop:n-cosmooth local}, which describes a Zariski-local presentation for $n$-cosmooth modules analogous to the presentation of a Cartier module by structure equations.  This is applied in Proposition \ref{proposition: lift n-cosmooth module} to lift $n$-cosmooth modules to Cartier modules of formal Lie groups.  We then combine this result with Cartier theory and descent to describe the essential image of $\M_n$ in §\ref{subsection: essential image}.

\subsection{Review of Cartier theory} \label{subsection: Cartier}

The \emph{$p$-typical Cartier ring} $\mathbb{E} = \mathbb{E}_R$ is the $V$-adic completion of the ring $D_R$; see \cite[Definition and Theorem 4.17]{Zin84}.
By definition, an $\mathbb{E}$-module $M$ is \emph{$V$-reduced} if $V: M \to M$ is injective and $M = \varprojlim_n M/V^nM$ (i.e. $M$ is $V$-adically complete and $V$-torsion free), and $M$ is \emph{$\infty$-cosmooth} if $M$ is $V$-reduced and $M/VM$ is a finite locally free $R$-module.  As usual, we define the \emph{rank} of an $\infty$-cosmooth module $M$ to be the rank of $M/VM$ as an $R$-module.

The following is a reformulation of the second main theorem of local Cartier theory, following the exposition of \cite[Chapter IV]{Zin84}.  Because the formulation is nonstandard, we give a proof with references to the literature.

\begin{theorem}[Cartier theory] \label{thm:Cartier theory}
The functor $G \mapsto \Hom(G^\vee, W)$ defines an equivalence of categories between formal Lie groups over $R$ and $\infty$-cosmooth $\mathbb{E}_R$-modules. The dimension of $G$ as a formal Lie group agrees with the rank of $\Hom(G^\vee,W)$ as an $\infty$-cosmooth module.
\end{theorem}

\begin{proof}
The standard formulation of Cartier theory gives us an equivalence by sending a formal Lie group $G$ to the $\mathbb{E}$-module of $p$-typical curves in $G$; moreover, under this equivalence, the dimension of $G$ agrees with the rank of its Cartier module (\cite[Theorem 4.23]{Zin84}). By \cite[Theorem 4.15]{Zin84}, we may identify the $p$-typical curves in $G$ with the $\mathbb{E}$-module $\Hom(\widehat{W}, G)$ where $\widehat{W}$ is the Witt vector formal group and the left $\mathbb{E}$-module structure arises via the identification $\mathbb{E} = \End(\widehat{W})^{\op}$ (\cite[Corollary 4.16]{Zin84}
\footnote{The identification in \cite[Corollary 4.16]{Zin84} is an identification of $\mathbb{E}$ with $\End(\widehat{W})^{\op}$. This is also consistent with the definition of the Cartier ring in \cite[Definition 3.6]{Zin84}.}). Since $\widehat{W}$ is the Cartier dual of $W$ \cite[Chapter V, \S 4.4.6]{DG70}, we find a natural isomorphism $\Hom(\widehat{W},G) \cong \Hom(G^\vee, W)$, which finishes the proof.
\end{proof}

\subsection{Preliminaries on modules over the Cartier ring}

For all $n \geq 1$, Lemma \ref{lemma: VD is a two-sided ideal} implies that $V^nD_R$ is a two-sided ideal. Thus, $\mathbb{E}_n := \mathbb{E}/V^n\mathbb{E} = D_R/V^nD_R$ is a ring, known as the \emph{truncated Cartier--Dieudonn\'e ring} over $R$.

\begin{lemma}\label{lemma: V is injective on E}
Left multiplication by $V: \mathbb{E} \to \mathbb{E}$ is injective.
\end{lemma}

\begin{proof}
By \cite[Definition/Theorem 4.17]{Zin84}, any element of $\mathbb{E}$ has a unique representation of the form
\begin{equation*}
\sum_{r, s \geq 0} V^r [x_{r, s}] F^s
\end{equation*}
with $x_{r, s} \in R$ and $x_{r, s} = 0$ for fixed $r$ and $s$ sufficiently large, and the lemma follows immediately. 
\end{proof}

The following lemma tells us that the rings $\mathbb{E}_n$ share some formal properties with $n$-cosmooth $D_R$-modules.

\begin{lemma} \label{lem:mod}
The rings $\mathbb{E}_n$ satisfy $V^i\mathbb{E}_n = \mathbb{E}_n[V^{n-i}]$ for all $i \in \{1, \dots, n-1\}$.  Multiplication by $V^{n-i}$ is an isomorphism $(\mathbb{E}_n/V^i\mathbb E_n)|_{\sigma^{n-i}} \cong V^{n-i}\mathbb{E}_n$ of left $\mathbb{E}_n$-modules.
\end{lemma}

\begin{proof}
Clearly $V^i \mathbb{E}_n \subset \mathbb{E}_n[V^{n-i}]$.  Conversely, if $x \in \mathbb{E}_n[V^{n-i}]$, lift $x$ to $\tilde x \in \mathbb{E}$; then $V^{n-i}\tilde x = V^n y$ for some $y \in \mathbb E$. By Lemma \ref{lemma: V is injective on E}, left multiplication by $V$ is injective on $\mathbb E$, so $\tilde x = V^i y$ and thus $x \in V^i \mathbb{E}_n$. Thus $V^i\mathbb E_n = \mathbb E_n[V^{n-i}]$. The second statement then follows, as
\begin{equation*}
(\mathbb{E}_n/V^i\mathbb E_n)|_{\sigma^{n-i}} = \mathbb{E}_n|_{\sigma^{n-i}}/V^i = \mathbb{E}_n|_{\sigma^{n-i}}/\ker(V^{n-i}) = V^{n-i}\mathbb{E}_n. \qedhere
\end{equation*}
\end{proof}

The following lemmas establish that certain properties of $\mathbb{E}_n$-modules can be checked modulo $V$.

\begin{lemma} \label{lem:Einj}
Let $u: T \to L$ be a map of free $\mathbb{E}_n$-modules.  If the induced map $\ol{u}: T/VT \to L/VL$ is injective, then so is $u$.
\end{lemma}

\begin{proof}
We proceed by induction. Suppose the lemma holds for $\mathbb{E}_n$-modules, and let $u: T \to L$ be a map of free $\mathbb{E}_{n+1}$-modules which is injective modulo $V$.  Then $u': T/V^{n}T \to L/V^{n}L$ is injective by the inductive hypothesis.  Furthermore there is a commutative square
\[\begin{tikzcd}
	{(T/VT)|_{\sigma^n}} & {(L/VL)|_{\sigma^n}} \\
	{V^nT/V^{n+1}T} & {V^nL/V^{n+1}L}
	\arrow["u", from=1-1, to=1-2]
	\arrow["{V^n}"', from=1-1, to=2-1]
	\arrow["{V^n}"', from=1-2, to=2-2]
	\arrow["u", from=2-1, to=2-2]
\end{tikzcd}\]
where the vertical arrows are the isomorphisms of Lemma \ref{lem:mod}. It follows that $V^{n}T \to V^{n}L$ is injective, and consequently $u$ is injective.
\end{proof}

Let $W(R)\{V\}$ denote the associative ring generated by $W(R)$ and $V$ subject to $aV = V\sigma(a)$ for all $a \in W(R)$. By construction there is a natural map $W(R)\{V\} \to D_R$.

\begin{lemma} \label{lem:generators}
Let M be an $\mathbb{E}_n$-module, and suppose that $e_1, \dots, e_r \in M$ are such that their images modulo $V$ generate $M/VM$ as an $R$-module. Then $M$ is generated over $W(R)\{V\}$ by $e_1,\ldots, e_r$.
\end{lemma}

\begin{proof}
As $e_1, \dots, e_r$ generate $M/VM$ over $R$, $V^i e_1, \dots, V^i e_r$ generate $V^iM/V^{i+1}M$ over $R$. Therefore $e_1, \dots, e_r$ generate $\gr_V(M)$ over $W(R)\{V\}$.  As $V$ is nilpotent on $M$, the $V$-adic filtration on $M$ is finite and exhaustive, so $e_1, \dots, e_r$ generate $M$ over $W(R)\{V\}$. 
\end{proof}
Finally, we show that for $n$-cosmooth modules, checking that a morphism is an isomorphism can be done modulo $V$.

\begin{lemma} \label{lemma: isomorphism mod V}
Let $f: M' \to M$ be a map of $n$-cosmooth $D_R$-modules such that the induced morphism $\ol{f}: M'/VM' \to M/VM$ is an isomorphism.  Then $f$ is an isomorphism.
\end{lemma}

\begin{proof}
Since $M$ and $M'$ are both $n$-cosmooth, for $i \leq n-1$ we have a commutative square 
\[\begin{tikzcd}
	{(M'/VM')|_{\sigma^i}} & {(M/VM)|_{\sigma^i}} \\
	{V^iM'/V^{i+1}M'} & {V^iM/V^{i+1}M}
	\arrow["f", from=1-1, to=1-2]
	\arrow["{V^i}", from=1-1, to=2-1]
	\arrow["{V^i}", from=1-2, to=2-2]
	\arrow["f", from=2-1, to=2-2]
\end{tikzcd}\]
where the vertical maps are isomorphisms. Since $f$ induces an isomorphism $M'/VM' \to M/VM$, $f$ also induces an isomorphism on $V$-adic associated graded, so $f$ is an isomorphism.
\end{proof}

\subsection{Lifting of Dieudonné modules}
In this subsection, we give a local presentation of $n$-cosmooth modules and apply this to prove that Zariski locally, we may lift any $n$-cosmooth module to an $\infty$-cosmooth module.

Recall from Definition \ref{definition: rank of n-cosmooth module} that the rank of an $n$-cosmooth $D_R$-module $M$ is defined to be the rank of $M/VM$ as an $R$-module. The following proposition then gives a Zariski-local description of $n$-cosmooth modules which may be thought of as a truncated version of the notion of structure equations for a Cartier module, c.f. \cite[Chapter IV.8]{Zin84}.

\begin{proposition}\label{prop:n-cosmooth local}
Let $M$ be a $D_R$-module.  Then $M$ is $n$-cosmooth of rank $r$ if and only if Zariski locally on $\Spec R$, there is a short exact sequence
\begin{equation*}
0 \to T \overset{u}\to L \to M \to 0
\end{equation*}
where $T$ and $L$ are free $\mathbb{E}_n$-modules of rank $r$ with bases $h_1, \dots, h_r$ and $g_1, \dots, g_r$, respectively, and 
\begin{equation*}
u(h_i) := Fg_i - \sum_{j = 1}^r a_{ij}(V) g_j
\end{equation*}
where $a_{ij}(V) \in W(R)\{V\}$.
\end{proposition}

\begin{remark}
    Note that the presentation $u(h_i) = Fg_i - \sum_{j=1}^r a_{ij}(V)g_j$ depends only on the image of $a_{ij}(V)$ in $\mathbb{E}_n$.
\end{remark}

Before proving Proposition \ref{prop:n-cosmooth local}, we establish the following lemma on injectivity of the map $u$ used to present an $n$-cosmooth module. 

\begin{lemma} \label{lem:inj}
Let $T$ and $L$ be free $\mathbb{E}_n$-modules of rank $r$ with bases $h_i$ and $g_i$ respectively, and consider a map $u: T \to L$ defined by $u(h_i) = Fg_i - \sum_j a_{ij}(V)g_j$, where $a_{ij}(V) \in W(R)\{V\}$.  Then $u$ is injective. 
\end{lemma}

\begin{proof}
By Lemma \ref{lem:Einj}, it suffices to consider $n = 1$. Then if $t = \sum_i r_i h_i$ is a general element of $T$ with $r_i \in R\{F\}$, we can write $u(t) = \sum_i c_i g_i$ with
\begin{equation*}
c_i = r_iF - \sum_j r_j a_{ji}. 
\end{equation*}
By considering the $F$-degree of the right-hand side, we find that $c_i = 0$ for all $i$ implies $r_i = 0$ for all $i$, and therefore $u: T \to L$ is injective.
\end{proof}

\begin{proof}[Proof of Proposition \ref{prop:n-cosmooth local}]
First, suppose that $M = L/u(T)$ where $u$ is as in the statement. Then $V^n = 0$ on $M$ since $M$ is an $\mathbb{E}_n$-module.  Moreover
\begin{equation*}
M/VM = \mathbb{E}_1^r/(F - (\ol{a_{ij}(V)})) \cong R^r
\end{equation*}
is a free $R$-module of rank $r$.  For all $i \in \{1, \dots, n-1\}$, Lemma \ref{lem:mod} implies $u: T/V^iT \to L/V^iL$ is injective, so applying the Snake lemma to multiplication by $V^i$ on the short exact sequence $0 \to T \to L \to M \to 0$ implies 
\begin{equation*}
M[V^i] = \frac{L[V^i]}{T[V^i]}.
\end{equation*}
But by Lemma \ref{lem:mod}, $L[V^i] = V^{n-i}L$, and the image of $V^{n-i}L$ in $M$ is $V^{n-i}M$, so $M[V^i] \subseteq V^{n-i}M$ and hence $M$ is $n$-cosmooth.

Conversely suppose that $M$ is $n$-cosmooth over $R$ of rank $r$. Working locally, assume that $M/VM$ is a free $R$-module. 
Let $e_1, \dots, e_r \in M$ be such that their images in $M/VM$ form a basis over $R$. By Lemma \ref{lem:generators}, $M$ is generated by $e_1,\ldots, e_r$ over $W(R)\{V\}$, so we can write
\begin{equation*}
Fe_i = \sum_j a_{ij}(V) e_j,
\end{equation*}
where $a_{ij}(V) \in W(R)\{V\}$. Let $T$ and $L$ be free $\mathbb{E}_n$-modules of rank $r$ with bases $h_i$ and $g_i$, respectively, and define $u: T \to L$ by $u(h_i) = Fg_i - \sum_j a_{ij}(V)g_j$. Lemma \ref{lem:inj} shows that the map $u$ is injective, and the implication proved above shows that $M' := \coker(u)$ is $n$-cosmooth.
By definition $g_i \mapsto e_i$ defines a map $f: M' \to M$. 
Note that $M'/VM'$ and $M/VM$ are free $R$-modules with bases $g_1,\ldots, g_r$ and $e_1,\ldots, e_r$, so that the induced map $\ol{f}: M'/VM' \to M/VM$ is an isomorphism.  Thus $f$ is an isomorphism by Lemma \ref{lemma: isomorphism mod V}, and we are done. 
\end{proof}

\begin{proposition}\label{proposition: lift n-cosmooth module}
Let $\ol{M}$ be an $n$-cosmooth $D_R$-module of rank $r$. Then Zariski locally on $\Spec R$ there exists an $\infty$-cosmooth $\mathbb{E}_R$-module $M$ of rank $r$ such that $M/V^nM \cong \ol{M}$.
\end{proposition}

\begin{proof}
Let 
\begin{equation*}
0 \to T \overset{u}\to L \to \ol{M} \to 0
\end{equation*}
be a local presentation of $\ol{M}$ as in Proposition \ref{prop:n-cosmooth local}. Let $T', L'$ be free $\mathbb{E}$-modules with bases $h_i'$ and $g_i'$, respectively.  Then we let $M$ be the $\mathbb{E}$-module defined by $M := \coker(u': T' \to L')$, where $u'(h_i') = Fg_i' - \sum a_{ij}(V)g_j'$.  By \cite[Lemma 4.37]{Zin84}, $T' \to L'$ is injective with $V$-reduced cokernel, so we conclude that the sequence
\begin{equation*}
0 \to T' \overset{u'}\to L' \to M \to 0
\end{equation*}
is exact and that $M$ is $V$-reduced. By definition $M/V^nM \cong \ol{M}$, so $M$ has rank $r$.
\end{proof}

\subsection{Essential image of the Dieudonné module functor} \label{subsection: essential image}

By Corollary \ref{corollary: image is n-cosmooth}, if $G/\Spec R$ is an $n$-cosmooth group scheme, then $\Mn(G)$ is an $n$-cosmooth $D_R$-module.  In this subsection, we show that every $n$-cosmooth $D_R$-module arises in this way.  Our strategy is to combine Proposition \ref{proposition: lift n-cosmooth module} with Cartier theory, and for that we make the following definition.

\begin{definition}
A group scheme $H$ over $R$ is said to be \emph{$\infty$-cosmooth} if it is Cartier dual to a formal Lie group over $R$, i.e. if we have an isomorphism $H \cong \underline{\Hom}(G, \mathbb{G}_m)$ for $G$ a formal Lie group over $R$.
\end{definition}

\begin{lemma}\label{lemma: ses for infty cosmooth}
    Let $H$ be an $\infty$-cosmooth group scheme.
    Then the short exact sequence $0 \to W|_{\sigma^n} \to W \to W_n \to 0$ induces a short exact sequence of $\mathbb{E}$-modules
\[\begin{tikzcd}
	0 & {\Hom(H,W|_{\sigma^n})} & {\Hom(H,W)} & {\Hom(H,W_n)} & 0
	\arrow[from=1-1, to=1-2]
	\arrow[from=1-2, to=1-3]
	\arrow[from=1-3, to=1-4]
	\arrow[from=1-4, to=1-5]
\end{tikzcd}\]
\end{lemma}
\begin{proof}
    To show the lemma, it suffices to show that $\Hom(H,W) \to \Hom(H,W_n)$ is surjective.
    Since $W = \varprojlim_m W_m$ as a group scheme, by induction it suffices to show $\Hom(H,W_{m+1}) \to \Hom(H,W_m)$ is surjective.
    Since $H^\vee$ is a formal Lie group, $H^\vee = \varinjlim_{r} H^\vee[F^r]$, so $H = \varprojlim_r H_r$ where $H_r$ is the cokernel of $V^r$.
	By Lemma \ref{lemma: maps into smaller Witt}, if $r \geq m$, then pullback along $H \to H_r$ induces an isomorphism
	\[ \Hom(H_{r}, W_m) \cong \Hom(H,W_m)\]
    By Corollary \ref{corollary: exact into Wn}, the map 
    $\Hom(H_{m+1},W_{m+1}) \to \Hom(H_{m+1},W_{m})$
    is surjective, and thus the map $\Hom(H,W_{m+1}) \to \Hom(H,W_m)$ is surjective.
\end{proof}

\begin{theorem} \label{theorem: essentially surjective}
The functor $\Mn$ on $n$-cosmooth group schemes has essential image the $n$-cosmooth $D_R$-modules.
\end{theorem}
\begin{proof}
    Let $\ol M$ be an $n$-cosmooth $D_R$-module. 
	By Proposition \ref{proposition: lift n-cosmooth module}, there is a Zariski open cover $U \to \Spec R$ and an $\infty$-cosmooth $\mathbb{E}_U$-module $M$ such that $M/V^nM \cong \ol M_U$.
	By Cartier theory (Theorem \ref{thm:Cartier theory}), $M = \Hom(H, W)$ for some $\infty$-cosmooth group scheme $H/U$.
	Define $G_U$ to be the cokernel of $V^n$ on $H$.
    Then $G_U/U$ is $n$-cosmooth, and by Lemma \ref{lemma: ses for infty cosmooth}, 
    \[\Hom(G_U,W_n) \cong \Hom(H,W_n) \cong \Hom(H,W)/V^n\Hom(H,W)= M/ V^nM \cong \ol{M}_U.\]
	Thus $G_U/U$ is an $n$-cosmooth group scheme such that $\Mn(G_U) \cong \ol{M}_U$. By Theorem \ref{theorem: isomorphism of unit}, $\Mn$ is fully faithful, so $G_U \cong \Gn\Mn(G_U) \cong \Gn(M_U)$. Since $\ol{M}_U$ descends along  $U \to \Spec R$, apply $\Gn$ to the descent data to obtain descent data for $G_U$. By Lemma \ref{lemma: finite type algebraic stack}, $\Sm_n$ is a stack, so $G_U$ descends to an $n$-cosmooth group scheme $G/\Spec R$ with $\Mn(G) \cong M$.
\end{proof}

\begin{remark}\label{remark: n-smooth locally as kernel of Fn}
    The proof of Theorem \ref{theorem: essentially surjective} shows that if $G$ is an $n$-smooth group scheme over $R$, then Zariski locally on $\Spec R$, there is a formal Lie group $\tilde G$ such that $\tilde G[F^n] = G$. 
\end{remark}

\subsection{Relation between $n$-smooth Dieudonné theory and Cartier theory}
\label{remark: interpolation with Cartier theory}
    The statement of Theorem \ref{thm:equiv} 
    formally implies the main theorem of ($p$-typical) Cartier theory (Theorem \ref{thm:Cartier theory}). 
    Recall from Lemma \ref{lemma: formal Lie} that 
    \[ 
        \Sm_\infty \simeq \varprojlim_n \Sm_n.
    \]
    It is also true that 
    \[ \{ \infty\text{-cosmooth } \mathbb{E}_R\text{-modules}\} \simeq \varprojlim_n \{\text{$n$-cosmooth $D_R$-modules}\}\]
    where the functor sends $M \mapsto \{M/V^nM\}_{n \geq 1}$ (using that $M$ is $V$-adically complete).
    If $\mathcal{G}$ is a formal Lie group, then
    \begin{align*}
    \Hom(\mathcal{G}^\vee, W) = \varprojlim_n \Hom(\mathcal{G}^\vee, W_n) = \varprojlim_n \Hom(\mathcal{G}[F^n]^\vee, W_n) 
\end{align*}
where the second equality uses Lemma \ref{lemma: maps into smaller Witt}.
Thus the functor $\Hom(-^\vee,W)$ on $\Sm_\infty$ is the limit of the functors of Theorem \ref{thm:equiv}. 

Full faithfulness of the $n$-smooth Dieudonn\'e functor (Theorem \ref{theorem: isomorphism of unit}) therefore gives a new proof that the Cartier functor is fully faithful on $\Sm_\infty$.
However, our calculation of the essential image of $\Hom(-^\vee,W_n)$ on $\Sm_n$ (Theorem \ref{theorem: essentially surjective}) makes use of Cartier theory in its proof.

\section{Base Change and Dieudonn\'e Module Computations} \label{section: base change}

With Theorem \ref{thm:equiv} established, in this section we turn our attention to the compatibility of our Dieudonn\'e theory with base change.  We also make some explicit computations of Dieudonn\'e modules over an arbitrary base.

\subsection{Base change}
Let $R \to R'$ be a map of $\F_p$-algebras. We write $\M_{n, R}$ (resp. $\M_{n, R'}$) for the level $n$ Dieudonn\'e functor over $R$ (resp. $R'$) of Definition \ref{definition: level n Dieudonne}.  If $G/\Spec R$ is an $n$-cosmooth group scheme, base change induces a natural map
\begin{equation} \label{equation: base change}
D_{R'} \otimes_{D_R} \M_{n, R}(G) \to \M_{n, R'}(G_{R'}),
\end{equation}
of $D_{R'}$-modules, where $G_{R'} = G \times_{\Spec R} \Spec R'$ is the base change of $G$ to $R'$. 

\begin{proposition} \label{proposition: base change}
Formation of the Dieudonn\'e module for $n$-cosmooth group schemes commutes with arbitrary base change.  More precisely, for any $n$-cosmooth $G/\Spec R$ and morphism $R \to R'$ of $\F_p$-algebras, the comparison morphism (\ref{equation: base change}) is an isomorphism. 
\end{proposition}
\begin{proof}
    By definition, $\G_n$ is compatible with base change (for arbitrary $D_R$-modules). By Theorem \ref{thm:equiv}, 
    \[ 
    \M_{n,R}: \{\text{$n$-cosmooth group schemes over $\Spec R$}\} \to \{\text{$n$-cosmooth $D_R$-modules}\}^{op}
    \]
    is an equivalence with inverse $\G_{n,R}$, so by applying $\G_{n, R'}$ to $(\ref{equation: base change})$, we find that $\M_n$ is also compatible with base change.
\end{proof}

\subsection{Dieudonn\'e module computations}

To illustrate the explicit nature of our Dieudonn\'e theory, we make some direct computations of Dieudonn\'e modules over an arbitrary base.

\begin{example} \label{example: Z/p}
Consider $(\Z/p^n\Z)_R$, an $n$-cosmooth group scheme over $R$ with $i$-cosmooth quotient equal to $\Z/p^i\Z$ for all $1 \leq i \leq n$. Writing $p^n = F^n \circ V^n$ on the Witt vectors, we have
\begin{equation*}
\ul{\Hom}(\Z/p^n\Z, W_m) = W_m[p^n].
\end{equation*}
Thus the Dieudonn\'e module of $\Z/p^n\Z$ over $R$ is $W_n(R)$ with the standard action of $F$ and $V$.  For $m \leq n$, we get the expected Dieudonn\'e module of the $m$-cosmooth quotient $\Z/p^m\Z$, while for $m > n$ we pick up extra contributions from nilpotents in the ring $R$. This explains why in our formulation of Dieudonn\'e theory, we must work with the truncated Witt vectors as opposed to the group ind-scheme $\varinjlim_n W_n$.
\end{example}

\begin{example}
    Let $R= \Fp[\lambda]$ and let $\widehat{\mathbb G}_\lambda$ be the formal Lie group over $R$ with group law $X,Y \mapsto X+Y+\lambda XY$. If $\lambda = 0$, the formal group $\widehat{\mathbb G}_\lambda$ is isomorphic to $\widehat{\mathbb G}_a$, while if $\lambda$ is invertible, $\widehat{\mathbb G}_\lambda$ is isomorphic to $\widehat{\mathbb G}_m$.
	Let $G_n = (\widehat{\mathbb G}_\lambda[F^n])^\vee$, identified with the cokernel of $V^n$ for $\widehat{\mathbb G}_\lambda^\vee$.
	Thus $G_n$ is an $n$-cosmooth group scheme over $\Spec R$. Its fiber at $\lambda = 0$ is isomorphic to $\alpha_{p^n}^\vee$, while its fiber over $\lambda=1$ is isomorphic to $\Z/p^n\Z$.

	We first compute the restricted Lie algebra of $G_1^\vee = \widehat{\mathbb G}_\lambda[F]$ by hand. Note that $\widehat{\mathbb G}_\lambda$ has invariant differential $dT/(1+\lambda T)$ and thus its Lie algebra is spanned by $(1+\lambda T)\frac{d}{dT}$.
	According to Hochschild's identity \cite[Lemma 1]{Hoc55}, if $\partial$ is a derivation of a commutative ring $A$ and $f \in A$, then $(f\partial)^p = f^p\partial^p + (f\partial)^{p-1}(f)\partial$.
	Thus 
	\[ 
		\left((1 + \lambda T)\frac{d}{dT}\right)^p = \left(\left((1 + \lambda T)\frac{d}{dT}\right)^{p-1}(1+\lambda T) \right)\frac{d}{dT} = \lambda^{p-1}(1+\lambda T)\frac{d}{dT} .
	\]
	We conclude that the Dieudonné module of $G_1$ is $R[F]/R[F](F - \lambda^{p-1})$.

	Now we calculate the Dieudonné module of $G_n$ for all $n \geq 1$. By \cite[Theorem 2.19.1]{SS01}, the group $\widehat{\mathbb G}_\lambda^\vee$ is the kernel of the epimorphism $F - [\lambda^{p-1}]: W \to W$.
	On account of the commutative diagram
\[\begin{tikzcd}
	0 & {(\widehat{\mathbb G}_\lambda^\vee)^{(p^n)}} && W && W & 0 \\
	0 & {\widehat{\mathbb G}_\lambda^\vee} && W && W & 0
	\arrow[from=1-1, to=1-2]
	\arrow[from=1-2, to=1-4]
	\arrow["{V^n}", from=1-2, to=2-2]
	\arrow["{F - [\lambda^{p^n(p-1)}]}", from=1-4, to=1-6]
	\arrow["{V^n}", from=1-4, to=2-4]
	\arrow[from=1-6, to=1-7]
	\arrow["{V^n}", from=1-6, to=2-6]
	\arrow[from=2-1, to=2-2]
	\arrow[from=2-2, to=2-4]
	\arrow["{F - [\lambda^{p-1}]}", from=2-4, to=2-6]
	\arrow[from=2-6, to=2-7]
\end{tikzcd}\]
	we find $G_n= \ker(F - [\lambda^{p-1}]:W_n \to W_n)$.
	Thus we have an exact sequence
	\[ 0 \to \Hom(W_n,W_n) \to \Hom(W_n,W_n) \to \Hom(G_n,W_n).\]
	The map $\mathbb{E}_n= \Hom(W_n,W_n) \to \Hom(G_n,W_n)$ reduced modulo $V$ is the corresponding map for $n=1$ by Lemma \ref{lemma: maps into smaller Witt}. Thus, $\mathbb{E}_n \to \Hom(G_n,W_n)$ is surjective modulo $V$, hence is surjective.
	The induced map $\Hom(W_n,W_n) \to \Hom(W_n,W_n)$ is right multiplication by $F - [\lambda]^{p-1}$.
	Thus the Dieudonné module of $G_n$ is 
	\[ \Mn(G_n) = \mathbb{E}_n / \mathbb{E}_n(F - [\lambda^{p-1}]).\]

	By Proposition \ref{proposition: base change}, we can calculate the Dieudonné modules of $\alpha_{p^n}^\vee$ and $\Z/p^n\Z$ through base change.
	At $\lambda = 0$, we find $\Mn(\alpha_{p^n}^\vee) = \mathbb{E}_n/\mathbb{E}_nF$. The underlying $W(R)$-module is $M = \bigoplus_{i=0}^{n-1} R|_{\sigma^i}$, and $V: M|_{\sigma} \to M$ acts as the shift.
	Similarly, for $\lambda = 1$, we find $G_n = (\mu_{p^n})^\vee = \Z/p^n\Z$, so $\Mn(\Z/p^n\Z) = \mathbb{E}_n/\mathbb{E}_n(F-1)$.
	The map $W_n \to \mathbb{E}_n/\mathbb{E}_n(F-1)$ is an isomorphism of $W_n\{V\}$-modules, and $F$ acts on the image of $W_n$ by 
	\[ F x = x^\sigma F \equiv x^\sigma \mod \mathbb{E}_n(F-1).\]
	We find $\Mn(\Z/p^n\Z) \cong W_n$ with the usual actions of $F$ and $V$, agreeing with Example \ref{example: Z/p}.
\end{example}

\section{Smoothness of Stacks and Truncation Morphisms} \label{section: smoothness}

Let $\tau_n$ denote the truncation $\Sm_{n+1} \to \Sm_n$.
For $n \geq 1$, Theorem \ref{thm:equiv} allows us to identify $\Sm_n^r$ with the moduli stack of $n$-cosmooth Dieudonn\'e modules of rank $r$.  Moreover, under these identifications, Lemma \ref{lemma: maps into smaller Witt} and Corollary \ref{corollary: exact into Wn} identify the truncation morphisms $\tau_n: \Sm_{n+1} \to \Sm_n$ with the truncations $M \mapsto M/V^nM$ on Dieudonn\'e modules. Theorem \ref{thm:smooth} follows from Propositions \ref{prop: smn is smooth} and \ref{prop: truncation is smooth}, which are applications of Proposition \ref{prop:n-cosmooth local}.

\begin{proposition}\label{prop: smn is smooth}
$\Sm_n$ is a smooth algebraic stack over $\Fp$.
\end{proposition}

\begin{proof}
Since $\Sm_n = \sqcup_{r \geq 1} \Sm_n^r$ and $\Sm_n^r$ is of finite type over $\Fp$ by Lemma \ref{lemma: finite type algebraic stack}, it suffices to show that $\Sm_n^r$ is formally smooth over $\Fp$ (\cite[Tag 0DP0]{stacks-project}). Let $A \to B$ be a surjection of local rings, and suppose that we are given an $n$-cosmooth rank $r$ module $M$ over $B$. Then by Proposition \ref{prop:n-cosmooth local} we can write 
\begin{equation*}
M \cong \frac{L}{(Fe_i - \sum_{j = 1}^r a_{ij}(V) e_j)}
\end{equation*}
where $L$ is a free $\mathbb{E}_{n, B}$-module with basis $e_1, \dots, e_r$ and the $a_{ij}(V) \in W(B)\{V\}$. Let $a_{ij}'(V)$ be lifts of the $a_{ij}(V)$ to $W(A)\{V\}$.  Then if $L'$ denotes a free $\mathbb{E}_{n, A}$-module with basis $e_1', \dots, e_r'$, Proposition \ref{prop:n-cosmooth local} gives that the module
\begin{equation*}
\frac{L'}{(Fe_i' - \sum_j a_{ij}'(V)e_j')}
\end{equation*}
is a lift of $M$ to an $n$-cosmooth rank $r$ module over $A$. 
\end{proof}

To prove smoothness of the truncation morphism $\tau_n: \Sm_{n+1} \to \Sm_n$, we need more control on the matrix of $F$ on a cosmooth module.
\begin{lemma}\label{lemma: multiplicative cosmooth coordinates}
    Let $M$ be an $n$-cosmooth $D_R$-module and $e_1,\ldots, e_r \in M$ such that their images form a basis over $R$ for $M/VM$. 
    If $[a] \in W(R)$ is the multiplicative lift of $a \in R$, then for all $m \in M$ there exist unique $a_i^j \in R$ for $i \in \{0,\ldots,n-1\}$ and $j \in \{1,\ldots, r\}$ such that
    \[ m = \sum_{i=0}^{n-1} \sum_{j=1}^r V^i[a_i^j] e_j.\]
\end{lemma}

\begin{proof}
    The proof is by induction on $n$. The statement holds for $n=1$.
    If $M$ is $n$-cosmooth, then $M/V^{n-1}M$ is $(n-1)$-cosmooth, so for $m \in M$ there exist unique $a_i^j \in R$ for $i < n-1$ such that 
    \[ m - \sum_{i=0}^{n-2}\sum_{j=1}^r V^i [a_i^j]e_j \in V^{n-1}M.\]
    Now left multiplication by $V^{n-1}$ defines an isomorphism of abelian groups $M/VM \to V^{n-1}M$, so there exist unique $a_{n-1}^j \in R$ such that
    \[ m - \sum_{i=0}^{n-2}\sum_{j=1}^r V^i[a_i^j]e_j = \sum_{j=1}^r V^{n-1}[a_{n-1}^j]e_j.\qedhere\]
\end{proof}

\begin{proposition}\label{prop: truncation is smooth}
The truncation morphism $\tau_n: \Sm_{n+1} \to \Sm_n$ induced by $G \mapsto G[F^n]$ is smooth and surjective. 
\end{proposition}

\begin{proof}
The surjectivity of $\tau_n$ follows from Proposition \ref{prop:n-cosmooth local}: over a field $k$, every $n$-cosmooth module $M$ can be written in the form 
\begin{equation*}
M \cong \frac{\mathbb{E}_{n, k}^r}{(Fe_i - \sum a_{ij}(V) e_j)}
\end{equation*}
with $a_{ij}(V) \in W(k)\{V\}$.  Then 
\begin{equation*}
 \frac{\mathbb{E}_{n+1, k}^r}{(Fe_i - \sum a_{ij}(V) e_j)}
\end{equation*}
is a lift of $M$ to an $(n+1)$-cosmooth module over $k$.

To show that $\tau_n$ is smooth, it suffices to show that the restriction $\tau_n: \Sm_{n+1}^r \to \Sm_n^r$ is smooth. Since this is a map between finite type stacks over $\Fp$, it is locally of finite presentation, so it suffices to show that $\tau_n$ is formally smooth (\cite[Tag 0DP0]{stacks-project}). Suppose we have a 2-commutative diagram 
\begin{equation} \label{eqn:formalsmooth}
\begin{tikzcd}
\Spec B \arrow[r] \arrow[hookrightarrow, d] & \Sm_{n+1}^r \arrow[d, "\tau_n"]  \\
\Spec A \arrow[r] \arrow[Rightarrow, ur, "\alpha"]& \Sm_n^r
\end{tikzcd}
\end{equation}
where $A \to B$ is a surjection of Artinian local rings.  Thus we are given an $(n+1)$-cosmooth module $N$ over $B$, an $n$-cosmooth module $\ol{M}$ over $A$, and an isomorphism $\alpha: \ol{M}|_{B} \to N/V^n N$.  
Pick $\{e_1,\ldots, e_r\} \subseteq \overline{M}$ and $\{f_1,\ldots,f_r\} \subseteq N$ such that $\{e_1,\ldots, e_r\}$ maps to a basis of $\ol{M}/V\ol{M}$ and $\alpha(e_j) \equiv f_j \mod V^n$ for all $j$.
By Lemma \ref{lemma: multiplicative cosmooth coordinates}, there exist unique $a_{ik}^j \in A$ and $b_{ik}^j \in B$ such that 
\[ Fe_k = \sum_{i=0}^{n-1}\sum_{j=1}^r V^i [a_{ik}^j] e_j, \qquad Ff_k = \sum_{i=0}^n \sum_{j=1}^r V^i [b_{ik}^j]f_j.\]
Since $\alpha(e_j) \equiv f_j \mod V^n$, 
the uniqueness of Lemma \ref{lemma: multiplicative cosmooth coordinates} implies that $a_{ik}^j \mapsto b_{ik}^j$ for $i < n$.
Pick $a_{nk}^j \in A$ such that $a_{nk}^j \mapsto b_{nk}^j$ under $A\to B$,
and define 
\[ M := \left.\bigoplus_{j=1}^r \mathbb{E}_{n+1,A}e_j \right/ \left( F e_k - \sum_{i=0}^n \sum_{j=1}^r V^i[a_{ik}^j]e_j\right).\]
Then $M/V^nM \cong \ol{M}$ and $M|_{B} \cong N$ compatibly with $\alpha$.
Thus $M$ defines a lift $\Spec A \to \Sm_{n+1}^r$ in the diagram (\ref{eqn:formalsmooth}).
\end{proof}

\begin{remark}
Note that the proofs of Propositions \ref{prop: smn is smooth} and \ref{prop: truncation is smooth} yield a stronger lifting property than what is needed for smoothness.  In particular, the Artinian hypothesis in the proof of Proposition \ref{prop: truncation is smooth} is not used.
\end{remark}

\printbibliography

@article{ALB23,
  author    = {Ansch{\"u}tz, Johannes and {Le Bras}, Arthur-César},
  doi       = {10.1017/fmp.2022.22},
  fjournal  = {Forum of Mathematics, Pi},
  journal   = {Forum Math. Pi},
  number    = {2},
  shorthand = {ALB23},
  title     = {Prismatic Dieudonn\'e Theory},
  volume    = {11},
  year      = {2023}
}

@misc{AM25,
  archiveprefix = {arXiv},
  author        = {Dima Arinkin and Joshua Mundinger},
  eprint        = {2512.13856},
  title         = {Cartier duality via Mittag-Leffler modules},
  year          = {2025}
}

@book{BBM82,
  author     = {Berthelot, Pierre and Breen, Lawrence and Messing, William},
  doi        = {10.1007/BFb0093025},
  fseries    = {Lecture Notes in Mathematics},
  pagination = {x+261},
  publisher  = {Springer},
  series     = {Lecture Notes in Math.},
  title      = {Th\'eorie de Dieudonn\'e Cristalline II},
  volume     = {930},
  year       = {1982}
}

@article{Car67Modules,
  author   = {Cartier, Pierre},
  fjournal = {Comptes Rendus Hebdomadaires des S\'{e}ances de l'Acad\'{e}mie des Sciences. S\'{e}ries A et B},
  issn     = {0151-0509},
  journal  = {C. R. Acad. Sci. Paris S\'{e}r. A-B},
  pages    = {129--132},
  title    = {Modules associ\'{e}s \`a un groupe formel commutatif. {C}ourbes typiques},
  volume   = {265},
  year     = {1967}
}

@book{Dem72,
  author    = {Demazure, Michel},
  fseries   = {Lecture Notes in Mathematics},
  publisher = {Springer},
  series    = {Lecture Notes in Math.},
  title     = {Lectures on {$p$}-divisible groups},
  volume    = {302},
  year      = {1972}
}

@book{DG70,
  author    = {Demazure, Michel and Gabriel, Pierre},
  fseries   = {Lecture Notes in Mathematics},
  publisher = {North-Holland Publishing Co., Amsterdam},
  shorthand = {DG70},
  subtitle  = {Tome I: G\'eom\'etrie Alg\'ebrique, G\'en\'eralit\'es, Groupes Commutatifs},
  title     = {Groupes alg\'ebriques},
  year      = {1970}
}

@article{Die55,
  author   = {Dieudonn\'{e}, Jean},
  doi      = {10.2307/2372633},
  fjournal = {American Journal of Mathematics},
  issn     = {0002-9327},
  journal  = {Amer. J. Math.},
  pages    = {429--452},
  title    = {Lie groups and {L}ie hyperalgebras over a field of characteristic {$p > 0$}. {IV}},
  volume   = {77},
  year     = {1955}
}

@misc{Dri24,
  archiveprefix = {arXiv},
  author        = {Drinfeld, Vladimir},
  eprint        = {2307.06194},
  title         = {On the {L}au group scheme},
  year          = {2024}
}

@misc{drinfeld-lectures,
  author = {Drinfeld, Vladimir},
  month  = {10},
  note   = {Lectures available on YouTube:
            \url{https://www.youtube.com/watch?v=tQFWrvR3j6g}, \\
            \url{https://www.youtube.com/watch?v=yZv1xYFIpBY}
            },
  title  = {The stacks of n-Truncated Barsotti-Tate Groups III-IV.},
  year   = {2023}
}

@book{Gro74,
  author    = {Grothendieck, Alexandre},
  note      = {\'{E}t\'{e}, 1970},
  number    = {45},
  publisher = {Les Presses de l'Universit\'{e} de Montr\'{e}al, Montreal, QC},
  series    = {S\'{e}minaire de Math\'{e}matiques Sup\'{e}rieures [Seminar on Higher Mathematics]},
  title     = {Groupes de {B}arsotti-{T}ate et cristaux de {D}ieudonn\'{e}},
  year      = {1974}
}

@article{Hoc55,
  author   = {Hochschild, Gerhard},
  doi      = {10.2307/1993043},
  fjournal = {Transactions of the American Mathematical Society},
  issn     = {0002-9947},
  journal  = {Trans. Amer. Math. Soc.},
  pages    = {477--489},
  title    = {Simple algebras with purely inseparable splitting fields of exponent {$1$}},
  volume   = {79},
  year     = {1955}
}

@article{Ill85,
  author    = {Illusie, Luc},
  booktitle = {Seminar on arithmetic bundles: the Mordell conjecture (Paris, 1983/84)},
  fjournal  = {Ast\'{e}risque},
  journal   = {Ast\'{e}risque},
  pages     = {151--198},
  title     = {D\'{e}formations de groupes de {B}arsotti-{T}ate (d'apr\`es {A}. {G}rothendieck)},
  volume    = {127},
  year      = {1985}
}

@article{Lau13,
  author   = {Lau, Eike},
  doi      = {10.1090/S0894-0347-2012-00744-9},
  fjournal = {Journal of the American Mathematical Society},
  journal  = {J. Amer. Math. Soc.},
  number   = {1},
  pages    = {129--165},
  title    = {Smoothness of the truncated display functor},
  volume   = {26},
  year     = {2013}
}

@book{Mes72,
  author    = {Messing, William},
  fseries   = {Lecture Notes in Mathematics},
  publisher = {Springer},
  series    = {Lecture Notes in Math.},
  title     = {The crystals associated to {B}arsotti-{T}ate groups: with applications to abelian schemes},
  volume    = {264},
  year      = {1972}
}

@book{MM74,
  author    = {Mazur, Barry and Messing, William},
  fseries   = {Lecture Notes in Mathematics},
  publisher = {Springer},
  series    = {Lecture Notes in Math.},
  title     = {Universal extensions and one dimensional crystalline cohomology},
  volume    = {370},
  year      = {1974}
}

@misc{Mon24,
  archiveprefix = {arXiv},
  author        = {Shubhodip Mondal},
  eprint        = {2405.12967v3},
  title         = {Dieudonn\'e theory via classifying stacks and prismatic $F$-gauges},
  year          = {2025}
}

@book{SGA3,
  editor     = {Demazure, Michel and Grothendieck, Alexandre},
  fseries    = {Lecture Notes in Mathematics},
  mrnumber   = {274458},
  publisher  = {Springer},
  series     = {Lecture Notes in Math.},
  shorthand  = {SGA3},
  subtitle   = {Tome I: {P}ropri\'et\'es g\'en\'erales des sch\'emas en groupes},
  title      = {Sch\'emas en groupes},
  titleaddon = {S\'eminaire de G\'eom\'etrie Alg\'ebrique du Bois Marie 1962/64 (SGA 3)},
  volume     = {151},
  year       = {1970}
}

@article{SS01,
  author   = {Sekiguchi, Tsutomu and Suwa, Noriyuki},
  doi      = {10.2748/tmj/1178207479},
  fjournal = {The Tohoku Mathematical Journal. Second Series},
  issn     = {0040-8735},
  journal  = {Tohoku Math. J. (2)},
  number   = {2},
  pages    = {203--240},
  title    = {A note on extensions of algebraic and formal groups. {IV}. {K}ummer-{A}rtin-{S}chreier-{W}itt theory of degree {$p^2$}},
  volume   = {53},
  year     = {2001}
}

@misc{stacks-project,
  author       = {{The Stacks project authors}},
  howpublished = {\url{https://stacks.math.columbia.edu}},
  shorthand    = {Stacks},
  title        = {The Stacks project},
  year         = {2024}
}

@article{TO70,
  author   = {Tate, John and Oort, Frans},
  fjournal = {Annales Scientifiques de l'\'{E}cole Normale Sup\'{e}rieure. Quatri\`eme S\'{e}rie},
  issn     = {0012-9593},
  journal  = {Ann. Sci. \'{E}cole Norm. Sup. (4)},
  mrnumber = {265368},
  pages    = {1--21},
  title    = {Group schemes of prime order},
  doi={10.24033/asens.1186},
  volume   = {3},
  year     = {1970}
}

@incollection{Wed01,
  author    = {Wedhorn, Torsten},
  booktitle = {Moduli of abelian varieties ({T}exel {I}sland, 1999)},
  isbn      = {3-7643-6517-X},
  mrnumber  = {1827029},
  pages     = {441--471},
  publisher = {Birkh\"{a}user, Basel},
  series    = {Progr. Math.},
  title     = {The dimension of {O}ort strata of {S}himura varieties of {PEL}-type},
  volume    = {195},
  year      = {2001}
}

@book{Zin84,
  author    = {Zink, Thomas},
  fseries   = {Teubner-Texte zur Mathematik [Teubner Texts in Mathematics]},
  publisher = {BSB B. G. Teubner Verlagsgesellschaft, Leipzig},
  series    = {Teubner-Texte Math.},
  title     = {Cartiertheorie kommutativer formaler {G}ruppen},
  volume    = {68},
  year      = {1984},
  note      = {Cited from the English translation, \emph{Cartier Theory of Commutative Formal Groups}, available at \url{https://perso.univ-rennes1.fr/matthieu.romagny/articles/zink.pdf}.}
}

\end{document}